\documentclass[11pt]{amsart}
\usepackage{amsmath, amssymb, verbatim}
\usepackage{pinlabel}
\usepackage{graphicx}
\newtheorem{thm}{Theorem}[section]
\newtheorem{cor}[thm]{Corollary}
\newtheorem{lem}[thm]{Lemma}

\newtheorem{prop}[thm]{Proposition}

\theoremstyle{definition}

\theoremstyle{remark}

\numberwithin{equation}{section}

\newcommand{\de}{\delta}

\newcommand{\ga}{\gamma}

\newcommand{\la}{\lambda}

\newcommand{\ka}{\kappa}

\newcommand{\Z}{\mathbb Z}

\newcommand{\del}{\partial}

\DeclareMathOperator{\tb}{tb}

\DeclareMathOperator{\rot}{rot}

\DeclareMathOperator{\Map}{Map}

\begin{document}

\title{On overtwisted, right-veering open books} 

\author{Paolo Lisca}
\address{Dipartimento di Matematica ``L. Tonelli'', Largo Bruno Pontecorvo~5, 
Universit\`a di Pisa, 56127 Pisa, ITALY} 
\email{lisca@dm.unipi.it}

\date{}

\begin{abstract} 
We exhibit infinitely many overtwisted, right--veering, non--destabilizable 
open books, thus providing infinitely many counterexamples to a conjecture of 
Honda--Kazez--Mati\'c. The page of all our open books is a four--holed sphere and 
the underlying 3--manifolds are lens spaces. 
\end{abstract}

\maketitle

\section{Introduction}\label{s:intro}

The purpose of this note is to construct infinitely many counterexamples to a conjecture 
of Honda, Kazez and Mati\'c from~\cite{HKM2}. For the basic notions of contact topology 
not recalled below we refer the reader to~\cite{Et0, Ge}. 

Let $S$ be a compact, oriented surface with boundary and $\Map(S,\del S)$ the group of 
orientation--preserving diffeomorphisms of $S$ which restrict to $\del S$ as the identity, 
up to isotopies fixing $\del S$ pointwise. An {\em open book} (a.k.a.~an {\em abstract open book}) is a 
pair $(S,\Phi)$ where  $S$ is a surface as above and $\Phi\in\Map(S,\del S)$. Giroux~\cite{Gi2} 
introduced a fundamental operation of {\em stabilization} $(S,\Phi)\to (S',\Phi')$ on open books, 
and proved the existence of a 1--1 correspondence 
between the set of open books modulo stabilization and the set of contact 3--manifolds 
modulo isomorphism (see e.g.~\cite{Et} for details). Honda, Kazez and Mati\'c~\cite{HKM} 
showed that a contact 3--manifold is tight if and only if it corresponds to an equivalence class 
of open books $(S,\Phi)$ all of whose monodromies $\Phi$ are {\em right--veering} 
(in the sense of~\cite[Section~2]{HKM}). In~\cite{Go, HKM} it is also showed that every 
open book can be made right--veering after a sequence of stabilizations. In~\cite{HKM2}, 
Honda, Kazez and Mati\'c proved that, when $S$ is a holed torus, the contact structure corresponding 
to $(S,\Phi)$ is tight if and only if $\Phi$ is right--veering, and conjectured that a 
non--destabilizable right--veering open book corresponds to a tight contact 3--manifold. 
The Honda--Kazez--Mati\'c conjecture was recently disproved by Lekili~\cite{Le}, who 
produced a counterexample $(S,\Phi)$ with $S$ equal to a four--holed sphere 
and whose underlying 3--manifold is the Poincar\'e homology sphere. 

We shall now describe our examples. Denote by $\de_\ga\in\Map(S,\del S)$ the class of a positive 
Dehn twist along a simple closed curve $\ga\subset S$. 

\begin{thm}\label{t:main}
Let $S$ be an oriented four--holed sphere, and $a, b, c, d, e$ 
the simple closed curves on $S$ shown in Figure~\ref{f:S}. 
\begin{figure}[ht]
\labellist
\small\hair 2pt
\pinlabel $a$ at -10 195
\pinlabel $b$ at 260 198
\pinlabel $c$ at -10 33
\pinlabel $d$ at 260 31
\pinlabel $e$ at 123 96
\endlabellist
\centering
\includegraphics[scale=0.5]{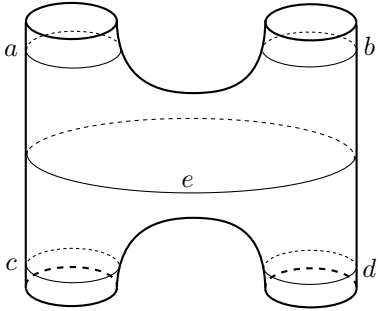}
\caption{The four--holed sphere $S$}
\label{f:S}
\end{figure}
Let $h, k\geq 1$ be integers. Define 
\[
\Phi_{h,k}:=\de_a^h\de_b\de_c\de_d\de_e^{-k-1}\in\Map(S,\del S).
\]
Then, 
\begin{itemize}
\item
The underlying 3--manifold $Y_{(S,\Phi_{h,k})}$ is the lens space 
\[
L((h+1)(2k-1)+2,(h+1)k+1);
\]
\item
the associated contact structure $\xi_{(S,\Phi_{h,k})}$ is overtwisted; 
\item
$\Phi_{h,k}$ is right--veering;
\item 
$(S,\Phi_{h,k})$ is not destabilizable. 
\end{itemize} 
\end{thm}

Warning: in the above statement we adopt the convention that the lens space $L(p,q)$ is the oriented 3--manifold 
obtained by performing a rational surgery along an unknot in $S^3$ with coefficient $-p/q$. 

We prove Theorem~\ref{t:main} in Section~\ref{s:proof}. The proof can be outlined as follows. 
In Proposition~\ref{p:csurgery} we use elementary arguments to determine a contact surgery presentation 
for the contact 3--manifold $(Y_{(S,\Phi_{h,k})},\xi_{(S,\Phi_{h,k})})$, and in Corollary~\ref{c:under} we apply  
Proposition~\ref{p:csurgery} and a few Kirby calculus moves to identify the underlying 3--manifold 
$Y_{(S,\Phi_{h,k})}$. In Proposition~\ref{p:ot} 
we appeal to calculations from~\cite{Le} to deduce that the contact Ozsv\'ath--Szab\'o invariant of 
$\xi_{(S,\Phi_{h,k})}$ vanishes, and we conclude from the fact that $Y_{(S,\Phi_{h,k})}$ is a 
lens space that $\xi_{(S,\Phi_{h,k})}$ must be overtwisted. We show  
that $\Phi_{h,k}$ is right--veering in Lemma~\ref{l:right-veer} by observing that this result follows directly 
from~\cite[Theorem~4.3]{AD}, but can also be deduced imitating the proof of~\cite[Theorem~1.2]{Le}, 
i.e.~applying~\cite[Corollary~3.4]{HKM}. Finally, 
we use results from~\cite{Ar, Le} to conclude that $(S,\Phi_{h,k})$ is not destabilizable. 

\smallskip
{\bf Acknowledgements:} I wish to thank Yanki Lekili for pointing out to me his paper~\cite{Le}. 
The present work is part of the author's activities within CAST, 
a Research Network Program of the European Science Foundation. 

\section{Proof of Theorem~\ref{t:main}}\label{s:proof}

Recall that every contact structure has a {\em contact surgery presentation}~\cite{DG}. We refer 
the reader to~\cite{DG} for the basic properties of contact surgeries, and to~\cite{LS} for the use 
of the `front notation'  in contact surgery presentations, in particular for the meaning of Figure~2 below. 

\begin{prop}\label{p:csurgery}
For $h, k\geq 1$, the contact structure $\xi_{(S,\Phi_{h,k})}$ has the contact surgery presentation 
given by Figure~\ref{f:csurg}.
\end{prop}
\begin{figure}[ht]
\labellist
\small\hair 2pt
\pinlabel $\frac{1}{k+1}$ at 216 181
\pinlabel $-\frac1h$ at 216 83
\endlabellist
\centering
\includegraphics[scale=0.5]{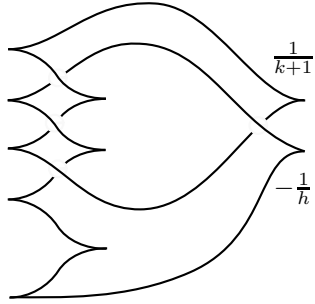}
\caption{Contact surgery presentation for $\xi_{(S,\Phi_{h,k})}$, $h,k\geq 1$.}
\label{f:csurg}
\end{figure}

\begin{proof}
Figure~\ref{f:obs}(a) represents an open book $(A,f)$, where $A$ is an annulus and $f$ is a positive 
Dehn twist along the core of $A$. The associated contact 3--manifold is the standard contact 3--sphere
$(S^3,\xi_{\rm st})$, the annulus $A$ can be viewed as the page of an open book decomposition of $S^3$, 
and the curve $\ka$ in the picture can be made Legendrian via an isotopy of the 
contact structure, in such a way that the contact framing on $\ka$ coincides with the framing induced on it by 
the page (see~e.g.~\cite[Figure~11]{Et}). The knot $\ka$ is the unique Legendrian unknot in $(S^3,\xi_{\rm st})$ having 
Thurston--Bennequin invariant $\tb(\ka)=-1$ and rotation number $\rot(\ka)=0$. A suitable choice of orientation for $\ka$ uniquely specifies 
its {\em negative} oriented Legendrian stabilization $\ka_-$, which satisfies $\tb(\ka_-)=-2$ and $\rot(\ka_-)=-1$. 
As shown in~\cite{Et}, $\ka_-$ can be realized as sitting on the page of a Giroux stabilization $(A',f')$ of $(A,f)$. This is illustrated 
in Figure~\ref{f:obs}(b), assuming the orientation on $\ka$ was taken to be ``counterclockwise'' in 
Figure~\ref{f:obs}(a).  Finally, Figure~\ref{f:obs}(c) shows an open book $(S,f'')$ obtained by Giroux 
stabilizing $(A',f')$ and containing both $\ka_-$ and $(\ka_-)_-$ in $S$ ($\ka_-$ was  
also given the ``counterclockwise'' orientation in Figure~\ref{f:obs}(b)). Clearly $(S, f'')$ still corresponds to $(S^3,\xi_{\rm st})$,   
and it is well--known that $\ka_-$, $(\ka_-)_-$ are the two Legendrian knots illustrated in Figure~\ref{f:csurg} 
(when oriented ``clockwise'' in that picture). By definition, $\Phi_{h,k}$ is obtained by pre--composing $f''$ 
with $k+1$ negative Dehn twists along parallel copies of $\ka_-$ and $h$ positive Dehn twists along parallel 
copies of $(\ka_-)_-$. Moreover, if $m\neq 0$ is an integer, $\frac 1m$--contact surgery along any Legendrian 
knot $\la$ is equivalent to $\frac{m}{|m|}$--contact surgeries along $|m|$ Legendrian push--offs of $\la$~\cite{DG}. 
Since page and contact framings coincide and by e.g.~\cite[Theorem~5.7]{Et} positive (negative, respectively) 
Dehn twists correspond to $-1$--contact surgeries ($+1$--contact surgeries, respectively), it is easy to check 
that the resulting contact structure is given by the contact surgery presentation of Figure~\ref{f:csurg}. 
\begin{figure}[ht]
\labellist
\small\hair 2pt
\pinlabel (a) at 89 5
\pinlabel (b) at 302 5
\pinlabel (c) at 562 5
\pinlabel {\scriptsize $+$} at 131 118
\pinlabel $\ka$ at 157 119
\pinlabel {\scriptsize $+$} at 328 161
\pinlabel {\scriptsize $+$} at 328 202
\pinlabel $\ka_-$ at 366 177
\pinlabel {\scriptsize $+$} at 589 148
\pinlabel {\scriptsize $+$} at 589 190
\pinlabel {\scriptsize $+$} at 589 287
\pinlabel $\ka_-$ at 602 256
\pinlabel $(\ka_-)_-$ at 670 200
\endlabellist
\centering
\includegraphics[scale=0.45]{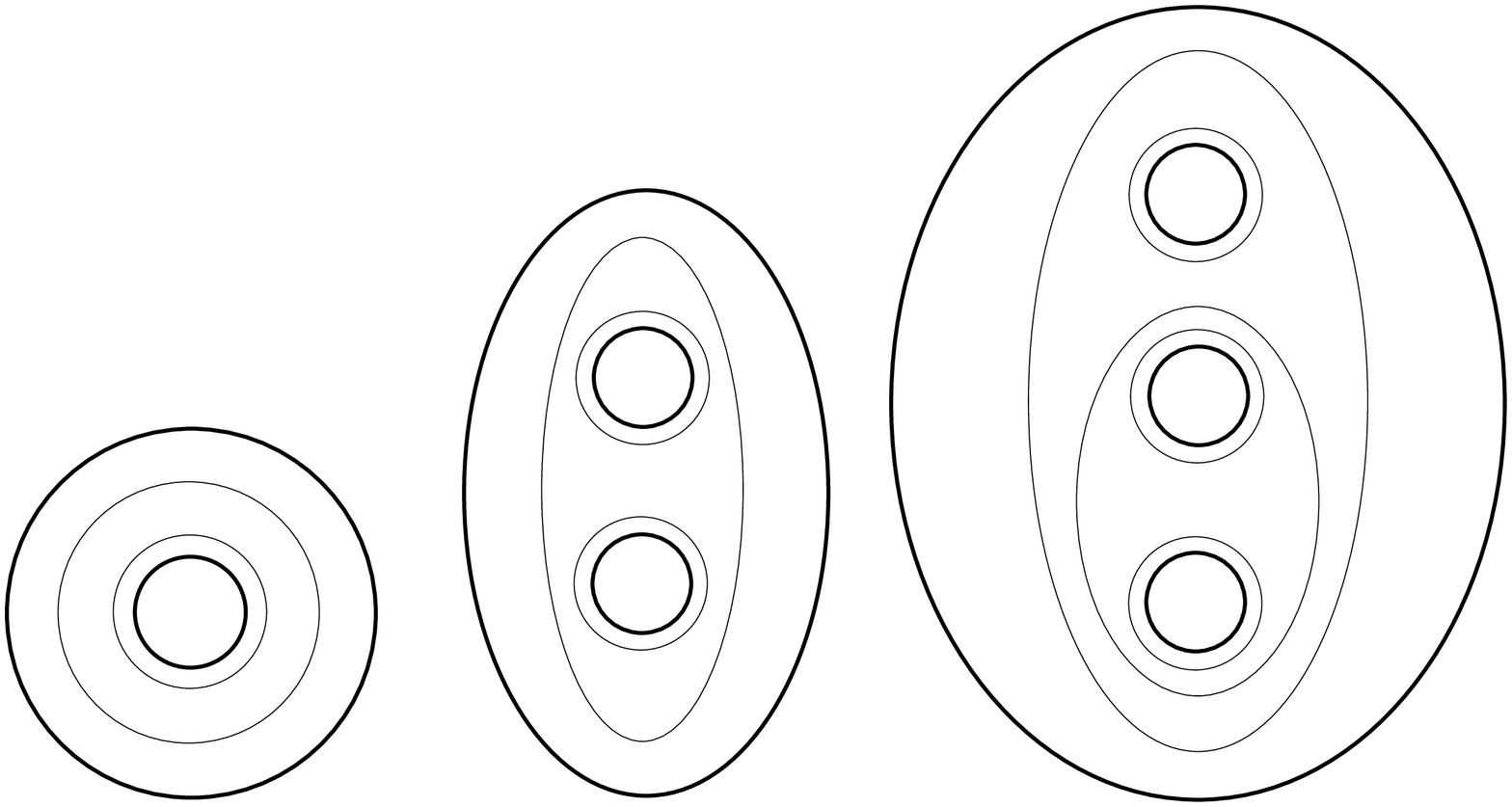}
\caption{Determination of the contact surgery presentation.}
\label{f:obs}
\end{figure}
\end{proof} 

\begin{cor}\label{c:under}
For $h, k\geq 1$, the oriented 3--manifold underlying the open book $(S,\Phi_{h,k})$ is the 
lens space $L((h+1)(2k-1)+2,(h+1)k+1)$.
\end{cor}

\begin{proof} 
Using the fact that the two Legendrian unknots illustrated in Figure~\ref{f:csurg} have Thurston--Bennequin 
invariants $-2$ and $-3$, it is easy to check that the topological surgery underlying Figure~\ref{f:csurg} is 
given by the first (upper left) picture of Figure~\ref{f:kirby}. 
\begin{figure}[ht]
\labellist
\small\hair 2pt
\pinlabel $-2$ at 49 95
\pinlabel {\scriptsize $-2+\frac1{k+1}$} at 20 68
\pinlabel {\scriptsize $-3-\frac1{h}$} at 78 68
\pinlabel {\scriptsize $-k-1$} at  172 68
\pinlabel {\scriptsize $0$} at 218 68 
\pinlabel {\scriptsize $1$} at 253 116 
\pinlabel {\scriptsize $1$} at 264 72
\pinlabel {\scriptsize $-1$} at 289 68 
\pinlabel {\scriptsize $h$} at 319 86
\pinlabel {\scriptsize $-2$} at 319 38
\pinlabel {\scriptsize $-k-1$} at 276 37
\pinlabel {\scriptsize $-1$} at 227 38
\pinlabel {\scriptsize $h$} at 205 29
\pinlabel {\scriptsize $-2$} at 143 35
\pinlabel {\scriptsize $-k-1$} at 115 35
\pinlabel {\scriptsize $-2$} at 85 35
\pinlabel {\scriptsize $-2$} at 55 35
\pinlabel {\scriptsize $-2$} at 32 35
\pinlabel {\scriptsize $h$} at 62 50
\pinlabel $\overbrace{\phantom{xxxxxxxxxx}}$ at 62 40
\endlabellist
\centering
\includegraphics[scale=1]{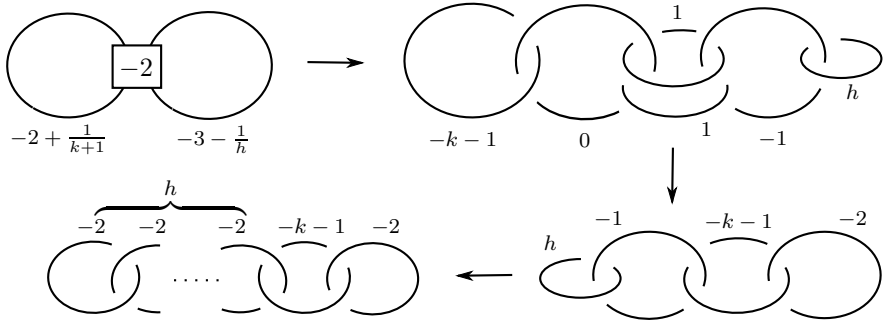}
\caption{Determination of the underlying 3--manifold.}
\label{f:kirby}
\end{figure}
Two $+1$--blowups and two inverse slam--dunks give the second picture, while the third picture 
is obtained from the second one by sliding the $-1$--framed knot over the $0$--framed knot and 
then applying two $+1$--blow--downs. The last picture is obtained simply converting 
the $h$--framed unknot in the third picture into the string of $-2$--framed unknots via a sequence 
of $-1$--blowups and a final $+1$--blowdown. The last picture shows that the underlying 
3--manifold $Y_{(S,\Phi_{h,k})}$ is obtained by performing a rational surgery on an unknot in $S^3$ with coefficient 
$-p/q$, where 
\[
\frac{p}q = 
2 - \cfrac{1}{k+1-\cfrac{1}{2-\cfrac{1}{\ddots -\cfrac{1}{2}}}}
=\frac{(h+1)(2k-1)+2}{(h+1)k+1}.
\]
Therefore, according to our conventions $Y_{(S,\Phi_{h,k})}$ can be identified with the lens space $L((h+1)(2k-1)+2,(h+1)k+1)$.
\end{proof} 

\begin{prop}\label{p:ot}
For $h, k\geq 1$, the contact structure $\xi_{(S,\Phi_{h,k})}$ is overtwisted. 
\end{prop}

\begin{proof}
By~\cite{Gi, Ho} a contact structure on a lens space is either overtwisted or Stein fillable. Moreover, 
Stein fillable contact structures have non--zero contact Ozsv\'ath--Szab\'o invariant~\cite{OSz}. 
Finally, \cite[Theorem~1.3]{Le} immediately implies that the contact invariant of 
$(S,\Phi_{h,k})$ vanishes, therefore $\xi_{(S,\Phi_{h,k})}$ must be overtwisted. 
\end{proof}

\begin{lem}\label{l:right-veer}
For $h, k\geq 1$, the diffeomorphism class 
\[
\Phi_{h,k}=\de_a^h\de_b\de_c\de_d\de_e^{-k-1}\in\Map(S,\del S)
\] 
is right--veering.  
\end{lem} 

\begin{proof} 
The lemma follows immediately from the statement of~\cite[Theorem~4.3]{AD}. Alternatively, one can imitate  
the proof of~\cite[Theorem~1.2]{Le}. Indeed,  applying~\cite[Corollary~3.4]{HKM} to the monodromy 
$\Phi_1=\de_e^{-k-1}$ and a properly embedded arc $\ga_{cd}\subset S$ 
disjoint from the curve $e$ and connecting the components $\del_c$ and $\del_d$ of $\del S$ parallel to 
the curves $c$ and $d$ shows that $\Phi_2=\de_d\de_e^{-k-1}$ is right--veering with respect 
to $\del_d$. Another application of the corollary to $\Phi_2$ and $\ga_{cd}$ shows that 
$\Phi_3=\de_c\de_d\de_e^{-k-1}$ is right--veering with respect to $\del_c$. Moreover, since $\de_c$ is right--veering 
with respect to $\del_c$ and the composition of right--veering diffeomorphisms is still right--veering~\cite{HKM}, 
$\Phi_3$ is right--veering with respect to $\del_d$ as well. Appying the corollary in the same way to $\Phi_3$ 
and an arc connecting the components of $\del S$ parallel to the curves $a$ and $b$ yields the statement 
of the lemma. 
\end{proof} 

\begin{proof}[Proof of Theorem~\ref{t:main}] 
Corollary~\ref{c:under}, Proposition~\ref{p:ot} and Lemma~\ref{l:right-veer} establish the first three portions 
of the statement. We are only left to show that $(S,\Phi_{h,k})$ is not destabilizable for every $h, k\geq 1$. 
If $(S,\Phi_{h,k})$ were destabilizable, it would be a stabilization 
of an open book $(S',\Phi')$, with $S'$ a three--holed sphere and $\Phi' = \tau_1^{a_1}\tau_2^{a_2}\tau_3^{a_3}$, 
where $a_i\in\Z$ and $\tau_i$ is a positive Dehn twist  along a simple closed curve parallel 
to the $i$--th boundary components of $S'$, $i=1,2,3$. 
By~\cite[Theorem~1.2]{Ar}, $\xi_{(S,\Phi_{h,k})}$ is tight if and only if $a_i\geq 0$, $i=1,2,3$. 
Therefore, by Proposition~\ref{p:ot} at least one of these exponents must be strictly negative. 
But the proof of~\cite[Theorem~1.2]{Le} shows that when one of the $a_i$'s is negative,   
any stabilization of $(S',\Phi')$ to an open book with page a four--holed sphere is not right--veering. 
This would contradict Lemma~\ref{l:right-veer}, therefore we conclude that 
$(S,\Phi_{h,k})$ cannot be destabilizable.
\end{proof} 

{\em Note:} after completing the first version of this paper the author was informed of 
independent, unpublished work of A.~Wand containing, in particular, a different proof of 
Proposition~\ref{p:ot}.

\end{document}